\documentclass[11pt,english]{article}
\usepackage[T1]{fontenc}
\usepackage[latin9]{inputenc}
\usepackage{textcomp}
\usepackage{amsmath}
\usepackage{amsthm}
\usepackage{amssymb}

\makeatletter

\newcommand{\noun}[1]{\textsc{#1}}

\theoremstyle{plain}
\newtheorem{thm}{\protect\theoremname}
  \theoremstyle{plain}
  \newtheorem{cor}[thm]{\protect\corollaryname}
  \theoremstyle{plain}
  \newtheorem{prop}[thm]{\protect\propositionname}
  \theoremstyle{plain}
  \newtheorem{lem}[thm]{\protect\lemmaname}
\newcommand{\lyxaddress}[1]{
\par {\raggedright #1
\vspace{1.4em}
\noindent\par}
}

\usepackage{mathrsfs}\usepackage{amsthm}\usepackage{amscd}

\usepackage{hyperref}
\usepackage{graphics}
\usepackage{a4wide}
\@ifundefined{definecolor}
 {\usepackage{color}}{}

\newcounter{EQNR}
\renewcommand{\contentsname}{Contents\\
{\footnotesize\normalfont(A table of contents should normally not
be included)}}

\makeatother

\usepackage{babel}
  \providecommand{\corollaryname}{Corollary}
  \providecommand{\lemmaname}{Lemma}
  \providecommand{\propositionname}{Proposition}
\providecommand{\theoremname}{Theorem}

\begin{document}

\title{A metric fixed point theorem and some of its applications}

\author{Anders Karlsson\footnote{The author was supported in part by the Swedish Research Council grant 104651320 and the Swiss NSF grants 200020-200400 and 200021-212864.}}

\date{January 11, 2023}
\maketitle
\begin{abstract}
A general fixed point theorem for isometries in terms of metric functionals
is proved under the assumption of the existence of a conical bicombing.
It is new even for isometries of Banach spaces as well as for non-locally
compact CAT(0)-spaces and injective spaces. Examples of actions on
non-proper CAT(0)-spaces come from the study of diffeomorphism groups,
birational transformations, and compact Kähler manifolds. A special
case of the fixed point theorem provides a novel mean ergodic theorem
that in the Hilbert space case implies von Neumann's theorem. The
theorem accommodates classically fixed-point-free isometric maps such
as those of Kakutani, Edelstein, Alspach and Prus. Moreover, from
the main theorem together with some geometric arguments of independent
interest, one can deduce that every bounded invertible operator of
a Hilbert space admits a nontrivial invariant metric functional on
the space of positive operators. This is a result in the direction
of the invariant subspace problem although its full meaning is dependent
on a future determination of such metric functionals. 
\end{abstract}

\section{Introduction}

Brouwer's fixed point theorem from 1912 as well as the infinite dimensional
extension that Schauder proved around 1930 have many applications.
These theorems require compactness, as Kakutani's 1943 examples of
fixed point free $(1+\epsilon)$-Lipschitz maps of the closed unit
ball in Hilbert spaces illustrated. For isometries there is an influential
paper \cite{BM48} by Brodskii and Milman proving conditions for the
existence of fixed points. For $1$-Lipschitz maps, also called nonexpansive
maps, there are fixed point theorems of Browder, Göhde and Kirk \cite{Br65a,Br65b,Go65,Ki65}
without compactness assumption but instead requiring certain good
properties of the Banach spaces. When the maps are uniformly strict
contractions, there is the simple but fundamental contraction mapping
principle, abstractly formulated in Banach's thesis from 1920, applicable
in any complete metric space and providing a unique fixed point. On
the other hand, for isometric maps Kakutani gave examples of the closed
unit ball in certain Banach spaces without fixed points (see below).
An example of Alspach \cite{Al81} showed moreover that for isometries
even with weak compactness of the convex set, there may not be any
fixed point. See the handbook \cite{Ha01} for more details on these
topics.  

The main purpose of this paper is to establish an entirely new type
of fixed point theorem. It could have been discovered decades ago
and is in several ways more general than existing fixed point theorems
for individual isometries. In particular it is new for isometries
of convex sets of Banach spaces. It has wide applicability as is testified
below with corollaries in ergodic theory, operator theory, geometric
group theory, dynamical systems, complex analysis, and other topics.
In some cases working out the exact consequences is beyond the scope
of this paper and may require further study.

Let $(X,d)$ be a metric space. A metric space is \emph{proper }if
every closed bounded set is compact. We call a map $T:X\rightarrow X$
\emph{isometric}, or \emph{isometric embedding},\emph{ }if it preserves
distances. It is an \emph{isometry }in case it has an (isometric)
inverse, which amounts to $T$ being a surjective isometric map. 

Let $x_{0}$ be a point of $X$. We map $X$ into the space $\mathrm{Hom}(X,\mathbb{R})$
of nonexpansive functions $X\rightarrow\mathbb{R}$ with the topology
of pointwise convergence, 
\[
\Phi:X\rightarrow\mathrm{Hom}(X,\mathbb{R})
\]
defined via
\[
x\mapsto h_{x}(\cdot):=d(\cdot,x)-d(x_{0},x).
\]
This is a continuous injective map and the closure $\overline{\Phi(X)}$
is compact. We call $\overline{X}:=\overline{\Phi(X)}$ the \emph{metric
compactification of }$X$ (following the terminology of \cite{Ri02}
for proper spaces) and we call its elements \emph{metric functionals}.
(The choice of topology here is different from the one taken in \cite{Gr81}
where it is uniform convergence on bounded sets. The limit functions
in this other topology are here called \emph{horofunctions}.)

The action of isometries on $X$ extends to an action on $\overline{X}$
by homeomorphisms, via
\[
(Th)(x)=h(T^{-1}x)-h(T^{-1}x_{0}).
\]
Alternatively one can consider everything up to additive constants.
This is related to the fact that the compactification is independent
of $x_{0}$. 

A bicombing is a selection of a geodesic between every two points
(\emph{distinguished geodesics}). More precisely, a \emph{weak conical}
\emph{bicombing }is a map $\sigma:X\times X\times[0,1]\rightarrow X$
such that $\sigma(x,y,\cdot)$ is a constant speed geodesic from $x$
to $y$ such that
\[
d(\sigma_{xy}(t),\sigma_{xy'}(t))\leq td(y,y')
\]
for all $x,y,y'\in X$ and $0\leq t\leq1$.

Examples of spaces admitting such bicombings include Banach spaces,
and convex subsets thereof, complete CAT(0)-spaces and injective metric
spaces. In addition, a particularly interesting specific example is
$\mathrm{Pos}$, the space of positive bounded linear operators of
a Hilbert space, explained more below, and on which every invertible
bounded linear operator acts by isometry. The $1$- Kantorovich-Wasserstein
distance on spaces of probability measures and the Weil-Petersson
metric on Teichmüller spaces are two further examples. 
\begin{thm}
\label{thm:main}Let $X$ be a metric space with a weak conical bicombing.
Let $T:X\rightarrow X$ be an isometric embedding. Then there is an
element $h\in\overline{X}$ such that 
\[
h(Tx)=h(x)-d
\]
 for all $x$ and where \emph{$d=\inf_{x}d(x,Tx)$. }In the case $h=h_{y}$
for some point $y\in X$ then $Ty=y.$ If $T$ is an isometry, then
\[
Th=h.
\]
\end{thm}

An inequality for nonexpansive maps was previously known from Gaubert
and Vigeral's paper \cite{GV12}, and even earlier from \cite{K01}
with a weaker but more general statement. For Banach spaces and in
terms of linear functionals an inequality is also proved in \cite{KoN81},
but the use of linear functionals instead of metric functionals gives
a strictly weaker result as can be inferred from an example by Mertens
in that same paper. The simple case of the theorem when $d=0$ was
observed by Gromov in \cite{Gr81}, see also Proposition \ref{prop:d=00003D0}
below. We call $h$\emph{ invariant }under $T$ when the first equality
in Theorem \ref{thm:main}, $h(Tx)=h(x)-d$, holds for all $x\in X$,
and we say that $h$ is a \emph{fixed point }of $T$ when $T$ acts
on the compactification $\overline{X}$ and $Th=h$. Note that some
convexity or bicombing property is necessary (as is the case for Brouwer,
Schauder et al) since a nontrivial rigid rotation of a circle has
no fixed point and the circle is a compact space, it means that there
are also no fixed metric functional. In contrast:
\begin{cor}
Any isometry $T$ of a metric space as in the theorem and with $\inf_{x}d(x,Tx)>0$
must have infinite order, in fact $T^{n}x\rightarrow\infty$. 
\end{cor}

For comparison, Basso proved a fixed point theorem for groups of isometries
of spaces with conical bicombings with a midpoint property, in a more
traditional sense and assuming compactness \cite{B18}. 

\paragraph*{\noun{Banach spaces}}

Here is a first application. Let $U$ be a norm-preserving linear
operator, that is $\left\Vert Uv\right\Vert =\left\Vert v\right\Vert $
for all $v$ in a Banach space $X$, where as usual all line segments
provide the required bicombing. In this setting the metric fixed point
theorem proves a new mean ergodic theorem: 
\begin{cor}
\label{cor:mean}\emph{(Mean ergodic theorem in Banach spaces)} Let
$U$ be a norm-preserving linear operator of a Banach space $X$ and
$v\in X$. Then there exists a metric functional $h$ such that
\[
\frac{1}{n}h\left(\sum_{k=0}^{n-1}U^{k}v\right)=-\inf_{x}\left\Vert Ux+v-x\right\Vert 
\]
for all $n\geq1.$ 
\end{cor}

In case of a Hilbert space one can deduce, in a very different way
indeed, the mean ergodic theorem proved by von Neumann using spectral
theory:
\begin{cor}
\label{cor:vonNeumann}\emph{(von Neumann, 1931) }Let $U$ be a unitary
operator of a Hilbert space $H$ and $v\in H$. Then the strong limit
\[
\lim_{n\rightarrow\infty}\frac{1}{n}\sum_{k=0}^{n-1}U^{k}v=\pi_{U}(v),
\]
where $\pi_{U}$ is the projection onto the $U$-invariant vectors. 
\end{cor}

(See section \ref{sec:Mean-ergodic-theorems} for the proofs.) The
mean ergodic theorem can thus be viewed as a fixed point theorem
The classical formulation (the convergence part in Corollary \ref{cor:vonNeumann})
is known to fail for general Banach spaces, such as $\ell^{\infty}(\mathbb{N})$
and $\ell^{1}(\mathbb{N})$ for shift operators, while Corollary \ref{cor:mean}
always holds true. An abstract characterization for when the classical
version holds is the content of the \emph{mean ergodic theorem in
Banach spaces} of Yosida-Kakutani in \cite{YK41}. There are many
papers on the topic of which operators are mean ergodic, see \cite{Ei10}
for a good exposition. See section \ref{sec:Mean-ergodic-theorems}
for a new mean ergodic theorems for power bounded operators.

Theorem \ref{thm:main} applies in particular to any closed convex
set $X$ of a Banach space, and provides fixed points in a generalized
sense. Note that it is known that the Mazur-Ulam theorem fails in
general, that is, not every isometry of a convex set is affine, see
the examples of Schechtman in \cite{Sch16}. We now relate Theorem
\ref{thm:main} to well-known conventionally fixed-point-free isometric
maps.

\textbf{Example. (Kakutani)} The map $T(x_{1},x_{2},...)=(1,x_{1},x_{2},...)$
is an isometric fixed-point-free map of the closed unit ball in the
sequence space $c_{0}$ with the supremum norm and all sequences tending
to $0$. The metric functional of the theorem for this map is $h(x_{1},x_{2},...)=\sup_{k}\left|x_{k}-1\right|-1$.

\textbf{Example. }The isometric map $T(x_{1},x_{2},...)=(1,x_{1},x_{2},...)$
of either $\ell^{1}(\mathbb{N})$ or $\ell^{2}(\mathbb{N})$, clearly
has no fixed point in the usual sense. On the other hand, in the $\ell^{1}(\mathbb{N})$
case it leaves the metric functional associated to $(1,1,1,...)$
invariant (\cite{GuK21})
\[
h(x_{1},x_{2},...)=\sum_{k=1}^{\infty}\left|x_{k}-1\right|-1,
\]
and in the $\ell^{2}(\mathbb{N})$ case the trivial metric functional
$h\equiv0$ is invariant. 

\textbf{Example. (Edelstein \cite{Ed64}) }A more sophisticated example
in the Hilbert space setting was given by Edelstein. It is an isometry
with unbounded but recurrent orbits. The trivial metric functional
is fixed. This map is an isometry also of $\ell^{1}(\mathbb{N})$
in which case the metric functional in the previous example is fixed
also for this map, as shown in \cite{Pe21}. In this context, note
that by \cite{Gu19}, the linear functional $f(x)\equiv0$ is not
a metric functional of $\ell^{1}(\mathbb{N})$. 

\textbf{Example. (Alspach \cite{Al81}) }This example\textbf{ }is
essentially what is called the Baker's transformation in ergodic theory.
The convex set $X$ in $L^{1}([0,1])$ consists of the integrable
functions taking values in $[0,2]$ of integral $1$. The isometric
map is $Tf(t)$ is $\min\{2,2f(2t)\}$ for $0\leq t\leq1/2$ and $\max\left\{ 0,2f(2t-1)-2\right\} $
for $1/2<t\leq1$. Gutièrrez noticed an invariant metric functional
of this map in \cite{Gu20}. 

\textbf{Example. (Shift in $\ell^{1}$) }Bader, Gelander, and Monod
\cite{BGM12} proved an unexpected fixed point theorem for $L^{1}$
spaces. In the traditional sense there may not be any fixed points
of isometries preserving a bounded closed convex set, but they showed
that there is a fixed point which may lie outside the set. An example,
mentioned in the introduction of \cite{BGM12}, is the shift map $T$
on $\ell^{1}(\mathbb{Z})$ and the convex bounded set being all the
non-negative functions whose values sum to $1$. The only fixed point
of this isometry is $0$. From the viewpoint of the present article,
given any two points $x$ and $y$, applying the shift enough times
on one of them, the $\ell^{1}$-distance $d(x,T^{n}y)$ is approaching
$\left\Vert x\right\Vert +\left\Vert T^{n}y\right\Vert =\left\Vert x\right\Vert +\left\Vert y\right\Vert $.
This shows that any orbit converges to the metric functional $h(x)=\left\Vert x\right\Vert =h_{0}(x),$
which is the metric functional of Theorem \ref{thm:main}. Thus $T0=0$
recovering the obvious fixed point in this case. 

\textbf{Example. (Basso \cite{B18}) }This is an example of an isometry
of a bounded complete convex subset of a strictly convex Banach space
without a fixed point. The map is the shift map $T$ on $\ell^{1}(\mathbb{Z})$
which is renormed with an equivalent norm that is strictly convex
and $T$ still an isometry. The closed convex hull $C$ of the vectors
which has $1$ in one coordinate and $0$ elsewhere has T as an isometry.
The only fixed point in the usual yense is clearly the zero-vector
which lies outside of $C$. Let $x_{0}$ be the vector that is $1$
in the $0$th coordinate and $0$ elsewhere. The metric functional
that is the limit of $h_{T^{n}x_{0}}$ as $n\rightarrow\infty$, 
\[
h(x)=\sqrt{(\left\Vert x\right\Vert _{1}+1)^{2}+\left\Vert x\right\Vert _{2}^{2}+1}-1
\]
is fixed by $T$.

\paragraph*{\noun{CAT(0)-spaces}}

When $X$ is a complete CAT(0)-space, geodesics are unique and provide
a conical bicombing almost directly from definitions. Such spaces
have a standard bordification adding equivalence classes of geodesic
rays, sometimes called the \emph{visual bordification,} or the \emph{visual
compactification} in case $X$ is proper, see \cite{BH99}. In Gromov's
choice of topology \cite{Gr81} on $\mathrm{Hom}(X,\mathbb{R})$,
the closure of $\Phi(X)$ above is the visual bordification \cite[Theorem 8.13]{Gr81,BH99}.
In particular, for nonproper $X$ the visual boundary may be empty,
but when proper it coincides with the metric compactification. Immediately
from Theorem \ref{thm:main} we therefore get:
\begin{cor}
\label{cor:cat0}Let $X$ be a complete CAT(0)-space. Then any isometry
of $X$ fixes a point in $\overline{X}$. In case $X$ is proper,
any isometry fixes a point in the visual compactification.
\end{cor}

In the proper case this is a standard fact based on the classification
of isometries, the elliptic part going back to E. Cartan. This is
of fundamental importance for the theory \cite{BGS85}. In the special
case of finite dimensional Cartan-Hadamard manifolds, the compactification
is a closed ball in this case and one could instead appeal to Brouwer's
fixed point theorem. The present paper provides a different proof
of this, and in fact the first equality in Theorem \ref{thm:main}
gives more precise information.

In the nonproper case, as explained in \cite{BH99} there are isometries
that have no fixed points in the visual bordification. For example
take $\ell^{2}(\mathbb{Z})$ and consider the shift and adding $1$
to the $0$th coordinate. Edelstein's isometry mentioned above has
$d=0$ (for example since the orbit does not tend to infinity, thus
$\tau(f)=0$, see below for the definition). This is no contradiction
since for Hilbert spaces the function $h\equiv0$ is a metric functional
\cite{K22,Gu19}. On the other hand, if $d>0$ the fixed metric functional
is nontrivial. See also Theorem \ref{thm:cat0} below for more information.
This is also the case for CAT(-1)-spaces in general since it is Gromov
hyperbolic and for such spaces it follows from Maher-Tiozzo \cite{MT18}
that $h\equiv0$ is \emph{not} a metric functional even without local
compactness. For isometries of CAT(-1)-spaces previous results can
be found in \cite[Theorem 5]{K01,KN04}. Claassens \cite{C20} has
completely determined the metric functionals for the infinite dimensional
real hyperbolic space.

Important examples of isometric actions on nonproper CAT(0)-space
are provided by the Cremona groups of birational transformations in
algebraic geometry \cite{Ca11,LU21}.

\paragraph*{\noun{Injective metric space}}

Metric space that are hyperconvex, or equivalently injective, have
properties which generalizes $L^{\infty}$ and trees. One important
feature of such spaces is that they admit a conical bicombing  as
shown by Descombes-Lang in \cite{DL15}. A result from 1979 by Sine
and Soardi shows that every nonexpansive self-map of a bounded injective
metric space has a fixed point, see \cite[Ch. 13]{Ha01}. Baillon
asked whether the boundedness could be relaxed to just assuming that
the orbit is bounded, but a counterexample was found by Prus \cite{Ha01}.

\textbf{Example (Prus)} Consider the hyperconvex Banach space $X=\ell^{\infty}(\mathbb{N})$
and the map $T:X\rightarrow X$ defined by $T((y_{n}))=(1+\lim y_{n},y_{0},y_{1},...)$
where the limit is an ultrafilter or Banach limit. This map is isometric
with bounded orbits but no fixed point in $X$. This is in contrast
with Theorem \ref{thm:main} that provides an invariant metric functional
for $T$. Moreover it must be nontrivial, since as shown below in
Proposition \ref{prop:Linfinity}, no metric functional for $\ell^{\infty}(\mathbb{N})$
is trivial, that is, the linear functional $f(x)\equiv0$ is not a
metric functional. 

Once again, Theorem \ref{thm:main} provides the missing fixed point
as it were: 
\begin{cor}
\label{cor:injective}Let $X$ be an injective metric space. Then
every isometry of $X$ fixes a point in $\overline{X}$. In case $X$
is proper, any isometry fixes a point in the visual compactification. 
\end{cor}

Details for how to deduce the second assertion of the corollary, which
is very easy, will be given below. A new topic in geometric group
theory is the subject of Helly groups, which are groups admitting
geometric isometric actions on injective metric spaces \cite{HuO21}.
Note that, although similar to the CAT(0)-case above, the boundaries
(metric vs visual) are now different even for proper spaces, the case
of $\ell^{\infty}$ illustrates this as is shown below. As Basso informed
me there is a fixed point theorem by U. Lang for isometry groups of
bounded orbits in \cite{La13}. For Busemann spaces one should mention
Navas paper \cite{N13} in this context.

\paragraph*{\noun{Diffeomorphims and biholomorphisms}}

It is well-known that the Weil-Petersson metric of the Teichmüller
space of a closed orientable surface has nonpositive curvature, is
not complete but is geodesically convex. Therefore:
\begin{cor}
Any element in the mapping class group of a genus g surface fixes
a metric functional of the corresponding Teichmüller space equipped
with the Weil-Petersson metric.
\end{cor}

Fixed points of mapping classes were already well-understood from
Thurston's compactification of Teichmüller spaces, using Brouwer's
theorem. 

The space of all Riemannian metrics $\mathrm{Met}(M)$ of a manifold
$M$ is a space first introduced by DeWitt in general relativity and
studied early on by Ebin in the 1960s. The diffeomorphisms group acts
on it by isometry.
\begin{cor}
Let $f$ be a diffeomorphism of a closed orientable compact manifold
$M$. Then $f$ fixes a metric functional of $\mathrm{Met}(M)$.
\end{cor}

\begin{proof}
It is known since some time that this space has features of nonpositive
curvature with a good description of what the geodesics are. The version
that applies best for us here is the theorem in \cite{Cl13} that
asserts that the metric completion is a CAT(0)-space. From this the
main fixed point theorem applies directly taking into account Lemma
\ref{lem:completion} below.
\end{proof}
A certain space of Kähler potentials of a connected compact Kähler
manifold was introduced by Calabi in the 1950s. It has a Weil-Petersson
type metric and the following holds for the same reasons as the previous
corollary:
\begin{cor}
Any complex automorphism of a compact Kähler manifold fixes a metric
functional of the associated space of Kähler potentials.
\end{cor}

The completed space which is a CAT(0)-space was used for progress
on a conjecture of Donaldson on the asymptotic behavior of the Calabi
flow, see \cite{Ba18} for a discussion and references about this
topic. The Calabi flow is in fact nonexpansive in this metric and
when there is no constant scalar curvature metric on the underlying
Kähler manifold, the flow may diverge. Note that the theorems in \cite{KM99,K01}
(compare with \ref{thm:cat0} below) apply to (the time-one map of)
the Calabi flow. These results provide a geodesic ray or metric functional
and may thus be relevant for Donaldson's conjecture. 

Holomorphic maps are always nonexpansive in the Kobayashi pseudo-metric
and thus biholomorphisms are isometries. The fixed point theorem in
this case will be treated in a forthcoming paper since some additional
arguments are needed in this case.

In these examples, from algebraic geometry, topology and complex geometry,
the metric fixed point theorem provides something new and nontrivial,
at least in the case that $d>0$, and the metric compactification
in these cases deserves to be studied. 

\paragraph*{\noun{The invariant subspace problem}}

The invariant subspace problem, open since the 1950s, asks whether
every bounded linear operator of a separable infinite dimensional
complex Hilbert space has a nontrivial closed invariant subspace.
In finite dimensions the statement follows for example from the fundamental
theorem of algebra, and von Neumman, who is one of the first to have
considered this problem, proved it in the case of compact operators.
Without loss of generality one may assume that the operator is invertible
otherwise one can add a multiple of the identity map, which does not
change the invariant subspaces.

The set of bounded positive operators $\mathrm{Pos}$ of a Hilbert
space $H$ has a natural metric which admits a weak conical bicombing.
Moreover every invertible bounded linear operator $g:H\rightarrow H$
acts by isometry on $\mathrm{Pos}$ via $p\mapsto gpg^{*}.$ Applying
Theorem \ref{thm:main} we get:
\begin{thm}
\label{thm:ISP}Every bounded, invertible linear operator $g$ of
a Hilbert space has an invariant nonconstant metric functional $h:\mathrm{Pos}\rightarrow\mathbb{R}$,
\[
gh=h.
\]
Moreover, the following equality holds:
\[
\inf_{p\in\mathrm{Pos}}\left|\log\sup_{v\neq0}\frac{(pg^{*}v,g^{*}v)}{(pv,v)}\right|=\lim_{n\rightarrow\infty}\frac{1}{n}\left|\log\left\Vert g^{n}g^{*n}\right\Vert \right|.
\]
\end{thm}

This presents a new approach to the invariant subspace problem and
a novel viewpoint of operator theory more generally. In finite dimensions
the statement above implies the existence of a nontrivial invariant
subspace as follows, either the operator is unitary with respect to
some scalar product or from the explicit determination of the metric
functionals in Lemmens' paper \cite{L21} there is an associated closed
linear subspace (the kernel of a semipositive operator), which is
nontrivial in case the dimension is at least two and that is invariant. 

The fact that the metric functional is nontrivial comes from the following
statement of independent interest. This is in contrast with linear
functionals and weak limits that can be just $0$.
\begin{prop}
Every metric functional of the space $\mathrm{Pos}$ is an unbounded
function. 
\end{prop}

In conclusion, Theorem \ref{thm:ISP} shows that a general bounded
linear operator does have some structure. Carleson once speculated
that the invariant subspace problem in the Hilbert space case perhaps
is a fixed point theorem on some Grassmannian \cite{CJ02}. Even though
this is not achieved in the present paper, the author believes that
the metric perspective on operator theory is of future interest more
generally. 

One suggestion could be to investigate the validity of an extension
of the Ryll-Nardzewski fixed point theorem for affine isometry groups
and see how it relates to Dixmier's unitarization problem. As Monod
pointed out to me years ago, a uniformly bounded representation of
a group corresponds to an isometric action with bounded orbits in
$\mathrm{Pos}$ and the equivalence to a unitary representation is
precisely the existence of a fixed point in $\mathrm{Pos}$ (that
fixed points may not exist, in contrast to CAT(0)-spaces, is also
coherent with Basso's example mentioned above). Some groups are non-unitarizable,
so once again the metric functionals is potentially an appropriate
device. This topics connects to the power-bounded mean ergodic theorem
in section \ref{sec:Mean-ergodic-theorems}, since the latter corresponds
to uniformly bounded representations of $\mathbb{Z}$ (or $\mathbb{N}$).

One area where spaces of positive operators have already been useful
is multiplicative ergodic theorems for operators. The result \cite[Corollary 7.1]{KM99}
is a substantial and surprising strengthening of an important theorem
of Ruelle \cite{R82} under an extra Hilbert-Schmidt condition that
allows using a submanifold of $\mathrm{Pos}$ that is CAT(0). A general
result is \cite[Theorem 1.7]{GK20} and in \cite{BHL20} several new
multiplicative ergodic theorems are obtained using finite traces and
a study of the corresponding spaces of positive operators (again CAT(0)
here). 

\textbf{Acknowledgements: }I thank Giuliano Basso, Tanja Eisner, Armando
Gutiérrez, Bas Lemmens, Alexander Lytchak, and Leonid Potyagailo for
several helpful comments.

\section{Metric preliminaries}

We consider the metric category, that is, objects are metric spaces
and morphims are nonexpansive maps, which means that
\[
d(f(x),f(y))\leq d(x,y)
\]
for all $x$ and $y$. The composition of nonexpansive maps remains
nonexpansive. Isometries are the isomorphisms in this category. 

One defines the \emph{minimal displacement 
\[
d=d_{f}=\inf_{x}d(x,f(x))
\]
 }and the \emph{translation number} 
\[
\ensuremath{\tau=\tau_{f}=\lim_{n\rightarrow\infty}d(x,f^{n}(x))/n}.
\]
 These numbers are analogs of the operator norm and spectral radius,
respectively, in particular note that $\tau\leq d.$ To see this,
since $\tau$ is independent of the point $x$ we may take a point
close to the infimal displacement in the sense that $d(x,f(x))<d+\epsilon.$
By the triangle inequality, $d(x,f^{n}(x))<n(d+\epsilon)$, so $\tau<d+\epsilon$
for any $\epsilon>0$ and hence the inequality. Rigid rotations of
a circle show that the inequality can be strict. On the other hand,
note that when Theorem \ref{thm:main} applies then necessarily $\tau=d$,
since by the theorem and in view of that metric functionals are nonexpansive,
\[
d(x,T^{n}x)\geq\left|h(x)-h(T^{n}x)\right|\geq nd.
\]
This was already known from \cite{GV12} and in the nonpositive curvature
case for isometries from Gromov. 

Let $\mathrm{Hom}(X,\mathbb{R})$ be the space of morphisms $X\rightarrow\mathbb{R}$
(\emph{fonctionnelles }in \cite{Ba32}) equipped with the topology
of pointwise convergence. Given a base point $x_{0}$ of the metric
space $X$, let 
\[
\Phi:X\rightarrow\mathrm{Hom}(X,\mathbb{R})
\]
be defined via
\[
x\mapsto h_{x}(\cdot):=d(\cdot,x)-d(x_{0},x).
\]
This is a continuous injective map and the closure $\overline{\Phi(X)}$
is compact, say the Arzela-Ascoli or Tychonoff's theorems. This is
the metric space analog of the Banach\textendash Alaoglu\emph{ }theorem.
We call $\overline{X}:=\overline{\Phi(X)}$ the \emph{metric compactification
of }$X$ and its elements \emph{metric functionals}. These are nonexpansive
and satisfy a \emph{metric Hahn-Banach theorem }\cite{K21a}. There
is a growing list of studied examples \cite{W18,KL18,HSW18,MT18,Gu19,C20,LP21,AK22}
to cite a few recent papers. It is easy to see that for Banach spaces
metric functionals are convex functions. In \cite{BCFS22} the same
topology is used as here for nonproper spaces, and that paper contains
a careful discussion of this construction. 

Busemann observed in the early part of the 20th century that any geodesic
ray $\gamma$ defines a metric functional: $b_{\gamma}(y)=\lim_{t\rightarrow\infty}d(y,\gamma(t))-t,$
called the \emph{Busemann function }of $\gamma$. This is akin to
how a vector defines a linear functional via a scalar product. Busemann
functions have been useful for a long time in the theories of spaces
with nonpositive respectively nonnegative curvature. The usefulness
of these metric notions without curvature conditions is illustrated
by the work described in \cite{KL11}.

Note the effect of changing base point $x_{0}$ to $y_{0}$:
\[
d(y,x)-d(x_{0},x)=d(y,x)-d(y_{0},x)+d(y_{0},x)-d(x_{0},x).
\]
It amounts in other words to just adding the constant $h_{x}(y_{0}).$

Isometries act by homeomorphisms of $\overline{X}$ via
\[
(Th)(x)=h(T^{-1}x)-h(T^{-1}x_{0}).
\]
Note that if for some metric functional $h$ one has that $h(Tx)=h(x)+c$
(or which is the same $h(T^{-1}x)=h(x)-c$) for all $x\in X$ with
a constant $c$, then $Th=h$. Conversely, having a fixed point, $h(x)=(Th)(x)=h(T^{-1}x)-h(T^{-1}x_{0})$
for all $x$, then the same equality holds with constant $c=h(T^{-1}x_{0})$.
As discussed in most references on the topic, from \cite{Gr81} to
\cite{BCFS22}, one could consider functionals up to an additive constant,
projectively, removing the role of any fixed based point. Some things
are more clear from this point of view, but functionals would then
just be defined up to an additive constant. 

Let $G$ be a group that fixes $h$, note that $T(g):=-h(gx_{0})=-d_{g}$
then defines a group homomorphism $T:G\rightarrow\mathbb{R}$. 

Bas Lemmens and I observed the following easy fact during a discussion
(essentially as recorded in \cite{K21a}), for isometries this was
pointed out by Gromov long time ago \cite{Gr81}: 
\begin{prop}
\label{prop:d=00003D0}Let $T$ be a nonexpansive map of a metric
space $X.$ Suppose that $d(T)=0.$ Then there is a metric functional
$h\in\overline{X}$ such that
\[
h(Tx)\leq h(x)
\]
 for all $x\in X$. In case $T$ is an isometric embedding then $h(Tx)=h(x)$
for all $x$, and when $T$ is an isometry $Th=h$.
\end{prop}

\begin{proof}
For any $\epsilon>0$ , take $y_{\epsilon}$ such that $d(y_{\epsilon},Ty_{\epsilon})<\epsilon$.
Take $x\in X$ and consider
\[
d(Tx,y_{\epsilon})-d(x_{0},y_{\epsilon})\leq d(Tx,Ty_{\epsilon})+d(Ty_{\epsilon},y_{\epsilon})-d(x_{0},y_{\epsilon})\leq\epsilon+d(x,y_{\epsilon})-d(x_{0},y_{\epsilon}).
\]
This shows that any limit point $h$ of $h_{y_{\epsilon}}$ as $\epsilon$
approaches $0$, which exists by compactness, has the property that
\[
h(Tx)\leq h(x).
\]
Now let us assume that $T$ is isometric, then 
\[
d(x,y_{\epsilon})-d(x_{0},y_{\epsilon})=d(Tx,Ty_{\epsilon})-d(x_{0},y_{\epsilon})\leq d(Tx,y_{\epsilon})+d(y_{\epsilon},Ty_{\epsilon})-d(x_{0},y_{\epsilon})
\]
\[
\leq\epsilon+d(Tx,y_{\epsilon})-d(x_{0},y_{\epsilon}),
\]
which in the limit shows that $h(x)\leq h(Tx).$ Thus $h(Tx)=h(x)$
and $Th=h$ if $T$ is an isometry.

Therefore $h(Tx)=h(x)$ for all $x$, and by a remark above $Th=h$
when $T$ is an isometry thus actually acting on $\overline{X}$. 
\end{proof}
Let us spell out a special case:
\begin{prop}
Let $T$ be an affine isometry of a bounded convex set $C$ of a reflexive
Banach space. Then there is a point $x\in C$ such that $Tx=x$.
\end{prop}

\begin{proof}
From Theorem \ref{thm:main} we have $d=\tau=0.$ Consider the sets
\[
\left\{ x:\left\Vert x-Tx\right\Vert <\epsilon\right\} .
\]
They are clearly convex when $T$ is affine and by reflexivity the
intersection of all $\epsilon>0$ is nonempty (Smulian). The points
in this intersection must be fixed points of $T$. 
\end{proof}
If one does not assume that the isometry is affine there is a celebrated
result of Maurey \cite{Ma81}, that states that in a weakly compact,
convex set of a superreflexive Banach space, every isometry has a
fixed point. 

\section{Proof of the metric fixed point theorem}

Fix $x_{0}\in X.$ We define the following family of contractions
$r_{s}:X\rightarrow X$, $0\leq s\leq1$ by $r_{s}(y)=\sigma_{x_{0}y}(s).$
Note that by the bicombing definition
\[
d(r_{s}(x),r_{s}(y))\leq sd(x,y)
\]
and $d(r_{s}(y),y)=(1-s)d(x_{0},y)$ from the constant speed geodesic
requirement.
\begin{lem}
\label{lem:completion}Let $X$ be a metric space and $Y$its metric
completion. Then $\overline{X}$ is canonically homeomorphic to $\overline{Y}.$
\end{lem}

\begin{proof}
We can consider $X$ as a subset of $Y$ in the natural way and select
a basepoint $x_{0}\in X$ for both spaces. First note that in a canonical
way $\mathrm{Hom}(X,\mathbb{R})\hookrightarrow\mathrm{Hom}(Y,\mathbb{R})$
since any element $\mathrm{h\in Hom}(X,\mathbb{R})$ has uniquely
defined values on $Y$. Indeed, for any sequence $x_{n}\rightarrow y$,
\[
\left|h(x_{n})-h(x_{m})\right|\leq d(x_{n},x_{m})
\]
which makes $h(x_{n})$ a Cauchy sequence of real numbers, hence converging,
which is the thus well-defined value $h(y)$. Similarly, by a $3\epsilon$-argument,
it follows that every limit point in $\mathrm{Hom}(X,\mathbb{R})$
is also a limit point in $\mathrm{Hom}(Y,\mathbb{R})$. This shows
that $\overline{X}\hookrightarrow\overline{Y}.$ 

Second, every element of $\mathrm{Hom}(Y,\mathbb{R})$ is also an
element of $\mathrm{Hom}(X,\mathbb{R})$ by restriction. This also
implies that any limit point $h$ in $\overline{Y}$ is also a metric
functional in $X$ by approximation and restriction. Indeed, say that
$h$ is a limit point of a set $\left\{ h_{y}\right\} ,$ then we
can approximate each $y$ by a Cauchy sequence in $x$. This shows
the opposite continuous inclusion and the lemma is established.
\end{proof}
In view of the lemma, and that we can also extend our maps $r_{s}$
for the same reasons since they are also nonexpansive maps, we may
for simplicity assume that $X$ is complete.

The map $Tr_{s}=T\circ r_{s}$ is a uniformly strict contraction if
$s<1$ since $T$ preserves distances. By the contraction mapping
principle we thus have a unique fixed point $y_{s}\in X$ for this
map.

Here is a lemma that applies even with $T$ is merely nonexpansive:
\begin{lem}
\label{lem:fixedpoint}Let $y_{s}$ be the unique fixed point of $T\circ r_{s}$
for $0\leq s<1$. Then
\[
d(x_{0},y_{s})\leq\frac{d(x_{0},Tx_{0})}{1-s}.
\]
 
\end{lem}

\begin{proof}
Let $D=d(x_{0},Tx_{0})$. Consider the ball of radius $R$ with center
$z$ 
\[
B_{R}(z)=\left\{ x\in X:d(x,z)\leq R\right\} .
\]
Note that from the bicombing 
\[
r_{s}(B_{R}(x_{0}))=B_{sR}(x_{0}).
\]
From the isometric property of $T$ obviously $T(B_{R}(x_{0}))\subseteq B_{R}(Tx_{0})$
for any $R$. Hence 
\[
Tr_{s}(B_{R}(x_{0}))\subseteq B_{sR}(Tx_{0}).
\]
Note that $B_{R}(Tx_{0})\subseteq B_{R+D}(x_{0})$ by the triangle
inequality. This means that if we take $R\geq D/(1-s)$ so that $sR+D\leq R$,
then 
\[
Tr_{s}(B_{R}(x_{0}))\subseteq B_{sR}(Tx_{0})\subseteq B_{sR+D}(x_{0})\subseteq B_{R}(x_{0}).
\]
From the contraction mapping principle, the iterates of the map $Tr_{s}$
converge towards the unique fixed point $y_{s}$. That is, for any
$x$, $(Tr_{s})^{n}x\rightarrow y_{s}$ as $n\rightarrow\infty$.
Since $B_{R}(x_{0})$ is an invariant set for $R=D/(1-s)$ we must
have that $y_{s}\in B_{R}(x_{0})$ as required.
\end{proof}
To get the first inequality we follow the proof of Gaubert-Vigeral
\cite{GV12} which even applies to nonexpansive maps $T$. For any
$x$ 
\[
h_{y_{s}}(x)-h_{y_{s}}(Tx)=d(x,y_{s})-d(Tx,y_{s})=d(x,y_{s})-d(Tx,T(r_{s}y_{s}))\geq d(x,y_{s})-d(x,r_{s}y_{s})
\]

\[
\geq d(x,y_{s})-d(x,r_{s}x)-d(r_{s}x,r_{s}y_{s})\geq d(x,y_{s})-d(x,r_{s}x)-sd(x,y_{s})=(1-s)d(x,y_{s})-d(x,r_{s}x)\geq
\]

\[
\geq(1-s)(d(x_{0},y_{s})-d(x_{0},x))-d(x,r_{s}x)=d(r_{s}y_{s},y_{s})-(1-s)d(x_{0},x)-d(x,r_{s}x)\geq
\]

\[
\geq d(Tr_{s}y_{s},Ty_{s})-(1-s)d(x_{0},x)-d(x,r_{s}x)=d(y_{s},Ty_{s})-(1-s)d(x_{0},x)-d(x,r_{s}x)\geq
\]

\[
\geq d(T)-(1-s)d(x_{0},x)-d(x,r_{s}x).
\]
The last two terms go to $0$ as $s\rightarrow1$. By compactness
there is a limit point of $h_{y_{s}}$ as $s$ approaches $1$ and
for any such limit point $h$ the above inequality shows that 
\[
h(x)-h(Tx)\geq d.
\]

So we have obtained $h$ with the property that $h(Tx)\leq h(x)-d$
for all $x\in X$. 

On the other hand, to get the other inequality we use that $T$ is
isometric,
\[
h_{y_{s}}(x)=d(x,y_{s})-d(x_{0},y_{s})=d(Tx,Ty_{s})-d(x_{0},y_{s})\leq
\]

\[
\leq d(Tx,y_{s})+d(y_{s},Ty_{s})-d(x_{0},y_{s})=d(Tx,y_{s})+d(Tr_{s}y_{s},Ty_{s})-d(x_{0},y_{s})=
\]

\[
=h_{y_{s}}(Tx)+d(r_{s}y_{s},y_{s})=h_{y_{s}}(Tx)+(1-s)d(x_{0},y_{s})\leq h_{y_{s}}(Tx)+(1-s)d(x_{0},Tx_{0})/(1-s)=
\]

\[
=h_{y_{s}}(Tx)+d(x_{0},Tx_{0}),
\]
where the last inequality comes from Lemma \ref{lem:fixedpoint}.
This means that for any $\epsilon>0$ we can find a metric functional
$h$ (using $x_{0}$ such that $d(x_{0},Tx_{0})<d+\epsilon$) such
that 
\[
h(x)\leq h(Tx)+d+\epsilon
\]
together with $h(x)\leq h(Tx)-d$ both inequalities holding for all
$x$. This applies to any limit point in particular the $h$ found
above. (Recall the above discussion about changing the base point,
it does not influence differences $h(Tx)-h(x)$). By compactness passing
to a limit point, there is a metric functional $h$ such that $h(x)\leq h(Tx)+d$
and $h(Tx)\leq h(x)-d$ for all $x\in X$, in other words 
\[
h(x)-d\leq h(Tx)\leq h(x)-d,
\]
hence equality for all $x\in X$. In case $T$ is an isometry, previous
remarks show that $Th=h$.

Suppose that $h=h_{y}$ for some point $y\in X$. Then $h_{y}(Tx)=h_{y}(x)-d,$
that is, 
\[
d(y,Tx)-d(y,x)=-d
\]
 for all $x$. Applying it to $x=y$, forces $d=0$ and $d(y,Ty)=0$,
since distances are positive or $0.$ By the standard axiom for metric
spaces it means that $Ty=y.$ 

With all this, Theorem \ref{thm:main} is proved.

\section{Mean ergodic theorems\label{sec:Mean-ergodic-theorems}}

In 1931 von Neumann proved the first ergodic theorem after being inspired
by remarks of Koopman and Weil.  Around the same time Carleman had
similar ideas, and as Eisner pointed out to me von Neumann subsequently
published a note analyzing Carleman's work on the topic. For references
and a good general discussion on this subject can be found in the
book by Eisner \cite{Ei10}. The mean ergodic theorem is proved in
virtually every book on ergodic theory. It is also well-known that
the operator version of the mean ergodic theorem fails in general
Banach spaces such as $\ell^{1}$ and $\ell^{\infty}$. From the metric
fixed point theorem, we deduce a new result that in contrast holds
in general: 
\begin{cor}
\emph{\label{cor:mean-1}(Mean ergodic theorem in Banach spaces)}
Let $U$ be a norm-preserving linear operator of a Banach space $X$
and $v\in X$. Then there exists a metric functional $h$ such that
\[
\frac{1}{n}h\left(\sum_{k=0}^{n-1}U^{k}v\right)=-\inf_{x}\left\Vert Ux+v-x\right\Vert 
\]
for all $n\geq1.$ 
\end{cor}

\begin{proof}
The line segments in the Banach space provide the required bicombing.
From the assumption $x\mapsto Ux+v$ is an isometric map and the sum
in the statement is the application of this map $n$ times to $x=0$.
The result now follows directly from Theorem \ref{thm:main} by iterating
the map. 
\end{proof}
In case of a Hilbert space one can deduce the usual mean ergodic theorem
in a quite different way to the standard proofs:
\begin{cor}
\label{cor:vonNeumann-1}Let $U$ be a unitary operator of a Hilbert
space $H$ and $v\in H$. Then the strong limit
\[
\lim_{n\rightarrow\infty}\frac{1}{n}\sum_{k=0}^{n-1}U^{k}v=\pi_{U}(v),
\]
where $\pi_{U}$ is the projection onto the $U$-invariant vectors. 
\end{cor}

\begin{proof}
Either the limit exists equalling $0$, or $\tau=d>0$ (see section
2) and Corollary \ref{cor:mean-1} gives a nontrivial metric functional,
which in this case for a Hilbert space must be linear, see \cite{K22}.
Hence there is a unit vector $w$ such that for all $n\geq1$ the
scalar product
\[
\left(\frac{1}{n}\sum_{k=0}^{n-1}U^{k}v,w\right)=d,
\]
where $d=\inf_{x}\left\Vert Ux+v-x\right\Vert $. Considering that
$U$ is norm-preserving, $\tau=d$, and $\left\Vert w\right\Vert =1$
it is a simple Hilbert space fact that this implies that the sum in
the statement converges strongly to $d\cdot w$ (see \cite{K22}).
From Theorem \ref{thm:main} we in addition have for any $x$ that
$(x-Ux-v,w)=\inf_{x}\left\Vert x-Ux-v\right\Vert =d$. Since the closed
linear span of vectors $\left\{ x-Ux\right\} $ is orthogonal to the
$U$-invariant vectors, the projection of $v$ onto the latter subspace
is $d\cdot w$. 
\end{proof}
A linear operator $A$ of a Banach space is \emph{power bounded} if
\[
\sup_{k>0}\left\Vert A^{k}\right\Vert <\infty.
\]
 Mean ergodic theorems have been considered for such operators for
example in the works of Kakutani and Yosida. Parallel to the above
discussion, the usual mean ergodic formulation fails in general but
there is now instead a metric replacement: 
\begin{cor}
\emph{(Power bounded ergodic theorem)} Let $A$ be a power bounded
operator of a Banach space $X$. Let
\[
\left\Vert x\right\Vert _{1}:=\sup_{k>0}\left\Vert A^{k}x\right\Vert 
\]
 for $x\in X$. Then for any $v\in X$ there is a metric functional
$h$ of $X$ with the norm $\left\Vert \cdot\right\Vert _{1}$ such
that
\[
h\left(\sum_{k=0}^{n-1}A^{k}v\right)\leq-n\tau
\]
for all $n>0$ and where $\tau=\inf_{x}\left\Vert Ax+v-x\right\Vert _{1}$. 
\end{cor}

\begin{proof}
Note that $\left\Vert \cdot\right\Vert _{1}$ is a norm and that $x\mapsto Ax+v$
is nonexpansive in the corresponding metric. The statement then follows
from the first part of the proof of the metric fixed point theorem. 
\end{proof}

\section{Some remarks on CAT(0), injective spaces, and $\ell^{\infty}$}

The following remarks are related to the main topic of this paper,
and although perhaps not new, some of them seems not to be stated
in the literature.

\textbf{\noun{CAT(0)-spaces}} Let $X$ be a proper CAT(0)-space. It
is shown in \cite[Theorem 8.13]{Gr81,BH99} that the metric compactification
is topologically equivalent to the visual compactification. Let $g$
be an isometry of $X$ with fixed point $h$ from Theorem \ref{thm:main}.
In case h is associated with a bounded sequences, then $h=h_{y}$
for some point $y\in X$ and then as shown above $gy=y$ (the usual
argument uses a circumcenter of the bounded orbit, see \cite[Corollary 2.8]{BH99}).
In the case $h\in\partial X$ then $h$ corresponds precisely to a
equivalence class of geodesic rays, and hence this is the fixed point
again in the traditional sense. Although this latter case is a simple
fact in view of the following argument: If the orbit of $g$ is unbounded,
let $g^{n_{k}}x_{0}$ be a convergent sequence to a boundary point
$\gamma.$ Note then that for any isometry $g'$ commuting with $g$
(for example $g$ itself) fixes this boundary point:
\[
g'(\gamma)=g'(\lim_{k}g^{n_{k}}x_{0})=\lim_{k}g'g^{n_{k}}(x_{0})=\lim_{k}g^{n_{k}}(g'(x_{0}))=\gamma,
\]
where the last equality comes from the fact that any two sequences
staying on bounded distance from each other (in the present case $d(g'x_{0},x_{0})$)
have the same accumulation points. Thus:
\begin{prop}
Let $X$ be a proper CAT(0)-space and $g$ an isometry with unbounded
orbits. The centralizer of $g$ fixes a point in the visual boundary
of $X$. 
\end{prop}

As noted previously, without properness, the metric compactification
is different from the visual bordification (which can be empty even
for unbounded spaces). Therefore Corollary \ref{cor:cat0} is of interest
for groups of isometries of nonproper spaces. Important examples of
isometric actions on nonproper CAT(0)-space are provided by the fundamental
Cremona groups of birational transformations of varieties. In the
two variable case this action was found by Manin and Zariski, and
greatly exploited for example by Cantat in \cite{Ca11}. More recently,
actions by the other Cremona groups by isometry on certain CAT(0)-cube
complexes were constructed by Lonjou and Urech \cite{LU21}. Fixed
points of individual isometries play a crucial role for these studies
and the translation length are related to algebraic notions. 

In this context, it may be worthwhile to point out the following:
\begin{thm}
\label{thm:cat0}Let $g$ be an isometry of a complete CAT(0)-space.
If $\inf_{x}d(x,gx)>0$, then $g$ has a fixed point in the visual
bordification. The metric functional corresponding to this point is
fixed and is a Busemann horofunction. It is moreover the unique horofunction
such that
\[
h(gx)=h(x)-d_{g},
\]
and for the geodesic ray $\gamma$ corresponding to $h$ it holds
that 
\[
\frac{1}{n}d(g^{n}x_{0},\gamma(d_{g}n))\rightarrow0.
\]
\end{thm}

\begin{proof}
First, it is known from \cite{BGS85}, but also reproved here in a
different way that $\tau_{g}=d_{g}$. Therefore $\tau_{g}>0$, and
a special case of the main result in \cite{KM99} then shows that
$g^{n}x_{0}$ converges to a visual boundary point $\gamma.$ In fact
the last assertion of the present theorem is proven and which is a
finer notion of convergence to the boundary. As shown above, we then
must have $g(\gamma)=\gamma$. This bordification is equivalent to
Gromov's horofunction bordification, thus $g$ fixes $h$, the metric
functional (which in this case is a Busemann horofunction $b_{\gamma}$)
associated to the geodesic ray emanating from $x_{0}$ representing
$\gamma$. As shown in \cite{KL11} the last assertion in the theorem
is equivalent to 
\[
-\frac{1}{n}h(g^{n}x_{0})\rightarrow\tau_{g}
\]
for $h=b_{\gamma}.$ Comparing Theorem \ref{thm:main} this is the
metric functional in that statement so there is an equality (which
actually is clear from remarks in section 2).
\end{proof}
Note that as alluded to above, shown for example in \cite{BH99},
in the the case of $\inf_{x}d(x,gx)=0$, there may not exist any fixed
point in the visual bordification. In case of CAT(0) cube complexes
there is however a more precise version (no parabolic isometries)
available due to Haglund \cite{H07}.

Some spaces are not complete but still have a conical bicombing, for
example the Weil-Petersson metric on the Teichmüller spaces since
it is geodesically convex with nonpositive curvature.

\textbf{\noun{Injective spaces}} The visual compactification was defined
in detail in \cite{DL15} (although since the choice of bicombing
is not unique it is apparently unclear how unique the visual compactification
is, while the metric compactification of course is unique). We need
to justify the second assertion of Corollary \ref{cor:injective}:
From Theorem \ref{thm:main} we have that either there a point $y\in X$
that is fixed be the isometry, or the isometry has unbounded orbits
in which case the small argument above for CAT(0)-spaces applies,
since any sequence staying on bounded distance from a convergent unbounded
sequence must converge to the same point as follows from the basic
properties of the visual compactification as in \cite{DL15}. This
proves that any isometry of a proper injective metric space has a
fixed point in the visual compactification.

\textbf{\noun{The $\ell^{\infty}$-spaces}} The $\ell^{\infty}$-space
is one of the standard injective spaces. Let us first consider it
on a finite set $X:=\ell^{\infty}(\left\{ 1,2,...,N\right\} =C(\left\{ 1,2,...,N\right\} )$.
In this case the visual compactification (as defined from the bicombing
of lines) is not equivalent to the metric one. To see this, first
notice that the rays can be represented by a unit vector $v$ and
$\gamma(t)=tv.$ Now for the metric functionals, which are already
known from papers by Walsh, Gutiérrez and others. Take $x_{0}=0$.
Since we can argue with sequences we take a sequence $y_{n}$ such
that $\left\Vert y_{n}\right\Vert \rightarrow\infty.$ By taking subsequences
we may divide the index set into $A$ and $B$ such that the coordinates
$y_{n}^{i}\geq0$ for all $n$ and $i\in A$ and $y_{n}^{j}<0$ for
all $n$ and $j\in B$. Moreover we can assume that the following
limits exist 
\[
y_{n}^{i}-\left\Vert y_{n}\right\Vert \rightarrow a^{i}\in[-\infty,0]
\]
for $i\in A$ and for $j\in B$ 
\[
-y_{n}^{j}-\left\Vert y_{n}\right\Vert \rightarrow b^{j}\in[-\infty,0].
\]
Denote by $A_{f}$ and $B_{f}$ those indices with a finite limiting
value. Since the norm needs to be realized on some coordinate we have
that $A_{f}\cup B_{f}$ is non-empty, indeed at least one of these
values is $0$. The corresponding metric functional which is the limit
of $h_{y_{n}}$ is now
\[
h(x)=\max\left\{ \sup_{i\in A_{f}}a^{i}-x^{i},\sup_{j\in B_{f}}b^{j}+x^{j}\right\} .
\]
These are the metric functionals of $X$ in addition to $h_{y}$.
For a ray $y_{n}=nv$ defined by a unit vector $v$ notice that $A_{f}=\left\{ i:v^{i}=1\right\} $
and $B_{f}=\left\{ j:v^{j}=-1\right\} .$ This means that the other
coordinates of $v$ play no role for its Busemann function. This shows
the difference between the two compactifications.

Determining the metric compactification of $\ell^{\infty}$ in infinite
dimensions seems not yet have been done, although there is an interesting
paper by Walsh \cite{W18} determining the so-called Busemann points
for Banach spaces (which lead him to a direct alternative proof of
the Mazur-Ulam theorem). Maybe one should keep in mind that the determination
of the linear dual space of $\ell^{\infty}$ in functional analysis
leads to signed finitely additive measures. 

Here is one interesting observation, which is in contrast to Hilbert
spaces but similar to a phenomenon Gutiérrez discovered for $\ell^{1}$
spaces \cite{Gu19} (also true for hyperbolic spaces \cite{MT18,C20,BCFS22}).
Actually, as Lytchak pointed out to me, this feature was established
for CAT(0)-spaces with finite telescoping dimension already in \cite[Lemma 4.9]{CL10}.
\begin{prop}
\label{prop:Linfinity}For $\ell^{\infty}$-spaces the linear functional
$f\equiv0$ is not a metric functional. In fact all metric functionals
are unbounded functions.
\end{prop}

\begin{proof}
Let $S$ be a set and consider $X=\ell^{\infty}(S)$. The finite dimensional
case was treated already. Any accumulation point of a bounded set
cannot be identically $0$ by evaluating on a point $x$ which lies
outside a ball centered at $0$ containing the bounded set, showing
that they all are unbounded functions. 

For the case of accumulation points of unbounded sets of metric functionals,
consider the set $C$ of $y\in X$ with $\left\Vert y\right\Vert >2.$
Take the subset $A$ of $X$ for which there is a coordinate $y^{s_{0}}$
such that $y^{s_{0}}>\left\Vert y\right\Vert -1/2>0.$ In the complement
of $A$ in $C$ there is instead a coordinate such that $-y^{s_{0}}>\left\Vert y\right\Vert -1/2>0.$ 

Look at $x\in X$ such that $-c$ for every $s\in S$, and some $c>1$.
Then for any $y\in A$ 
\[
h_{y}(x)=\sup_{s}|y^{s}-x^{s}|-\left\Vert y\right\Vert \geq y^{s_{0}}-(-c)-\left\Vert y\right\Vert >-1/2+c>1/2.
\]
For $y$ in the complement of $A$ in $C$ we instead look at $z\in X$
such that $z^{s}=c$ for all $s$ , some $c>1$, and have similarly
\[
h_{y}(z)=\sup_{s}|y^{s}-z^{s}|-\left\Vert y\right\Vert \geq-y^{s_{0}}+c-\left\Vert y\right\Vert >-1/2+c>1/2.
\]
Note that for any metric functional $|h(z)|\leq d(0,z)=1$. The two
sets $\left\{ f:f(x)<1/2\right\} $ and $\left\{ f:f(z)<1/2\right\} $
are neighborhoods of $f\equiv0$ in $\mathrm{Hom}(X,\mathbb{R})$.
The above two inequalities show that every cluster point of $\Phi(C)$
stays outside the intersection of these two neighborhoods, thus the
functional identically equal to $0$ is not in $\overline{X}$. Indeed,
we see that each $h$ takes values larger than $c-1/2$ but $c$ is
arbitrarily large, hence all metric functionals are unbounded.
\end{proof}

\section{The space $\mathrm{Pos}$}

Let $H$ be a Hilbert space and $\mathrm{Pos}=\mathrm{Pos}(H)$ the
set of positive definite symmetric (self-adjoint) bounded linear operators
on $H.$ This is an open cone in the Banach space of all symmetric
operators equipped with the operator norm (which by the spectral theorem
equals the spectral radius). This is the prototype space for spaces
of metrics discussed above concerning diffeomorphisms and biholomorphisms.
Therefore I expect that some things discussed below will be relevant
for the other contexts.

One can put a complete metric on the space $\mathrm{Pos}$, \emph{the
Thompson metric}, for $p,q\in Pos$
\[
d(p,q)=\sup_{v\in H,\left\Vert v\right\Vert =1}\left|\log\frac{(qv,v)}{(pv,v)}\right|,
\]
which is also a Finsler metric \cite{CPR94}. It is known from this
reference and subsequent papers that it has a weak conical bicombing
which follows from Segal's inequality: 
\[
\left\Vert \exp(u+v)\right\Vert \leq\left\Vert \exp(u/2)\exp(v)\exp(u/2)\right\Vert 
\]
for self-adjoint bounded operators and the operator norm. The bicombing
does not come from the lines in the cone even though these lines are
also geodesics, but instead the distinguished geodesics are the image
of lines in the tangent space $\mathrm{Sym}$ under the exponential
map (note that the exponential map in the sense of algebra and analysis,
coincides for this space with the notion in differential geometry).
We consider $\mathrm{Sym}$ with the distance that the operator norm
gives. The inequality shows that the exponential map is distance preserving
on lines, and metric noncontracting in general \cite{CPR94}. Note
that this discussion centered around the point $x=I$ but since $\mathrm{Pos}$
is a homogeneous space the same applies to any base point. Instead
of considering the full algebra of bounded linear operators one can
do the same constructions for $C^{*}$-algebras, see for example \cite{KM99}. 

Invertible bounded linear operators $g$ act on $\mathrm{Pos}$ by
isometry $p\mapsto gpg^{*}.$ This action is transitive. If $g$ fixes
a point $p$ it means it is a unitary operator in the corresponding
scalar product.

In finite dimensions there are several studies about the metric compactification
in the literature, in the standard Riemannian metric (of less interest
for us) \cite{BH99} and for various Finsler metrics \cite{KL18,HSW18,L21}.
For the application we have in mind, we need to go to infinite dimensions,
in this setting there is basically only \cite{W18}.

Here we now establish the following, which remedies a vexing property
of weak topologies in general in functional analysis that a weak limit
could just be $0$:
\begin{prop}
\label{prop:Posfunctional}The function identically $0$ is not a
metric functional of $\mathrm{Pos}$ and thus every metric functional
is nonconstant, in fact unbounded.
\end{prop}

\begin{proof}
First note that since $h(I)=0$ for any metric functional $h$, being
constant is the same as being $0$ everywhere. (Here we are taking
the identity operator to be the base point of $\mathrm{Pos}$.) 

Secondly, any cluster point of a bounded set of $h_{p}$ is clearly
an unbounded function: since taking $q$ with large norm much bigger
than the radius of the ball around $I$ that contains the set of $p$
under consideration, then
\[
h_{p}(q)=d(q,p)-d(I,p)\geq d(q,I)-2d(I,p)\gg0,
\]
and it shows the function is unbounded in $q$. (This applies to any
metric space.) 

We now consider the set of $p$ outside the ball of a certain sufficiently
large radius. We will show that this does not intersect the following
neighborhood of $0$ in $\mathrm{Hom}(\mathrm{Pos},\mathbb{R})$:
\[
N(0)=\left\{ f:\left|f(a)\right|<1/2\textrm{ and \ensuremath{\left|f(b)\right|<1/2}}\right\} 
\]
where $a=\exp(I)$ and $b=\exp(-I)$. 

To see this take $p=\exp(Y)\in\mathrm{Pos}$ with $\left\Vert Y\right\Vert >C>2$.
Consider:
\[
h_{p}(x)=d(x,p)-d(I,p)
\]
which is greater or equal to $h_{Y}(\log x)=\left\Vert \log x-Y\right\Vert -\left\Vert Y\right\Vert $
by the metric properties of the exponential map recalled above. For
a constant $\epsilon$ with $0<\epsilon<1/2,$ we take a unit vector
$v$ in $H$ such that $\left|(Yv,v)\right|>\left\Vert Y\right\Vert -\epsilon$
in view of a standard fact of symmetric operators. Now, in the case
$(Yv,v)>\left\Vert Y\right\Vert -\epsilon,$ we thus have

\[
h_{p}(b)\geq h_{Y}(-I)=\left\Vert -I-Y\right\Vert -\left\Vert Y\right\Vert \geq\left|(-Iv-Yv,v)\right|-\left\Vert Y\right\Vert =\left|-1-(Yv,v)\right|-\left\Vert Y\right\Vert 
\]
and assuming $C>1+\epsilon$ this expression equals
\[
1+(Yv,v)-\left\Vert Y\right\Vert \geq1-\epsilon>1/2.
\]
In the opposite case $-(Yv,v)>\left\Vert Y\right\Vert -\epsilon$,
we instead look at 
\[
h_{p}(a)\geq h_{Y}(I)=\left\Vert I-Y\right\Vert -\left\Vert Y\right\Vert \geq\left|(Iv-Yv,v)\right|-\left\Vert Y\right\Vert =\left|1-(Yv,v)\right|-\left\Vert Y\right\Vert 
\]
and assuming $C>1+\epsilon$ this expression equals
\[
1-(Yv,v)-\left\Vert Y\right\Vert \geq1-\epsilon>1/2.
\]
 In total this shows that the set of all functionals $h_{p}$ with
$\left\Vert p\right\Vert >\exp(2),$ does not intersect $N(0).$ It
follows that $f\equiv0$ cannot be a cluster point of the set $\Phi(\mathrm{Pos})$
and thus not a metric functional.

Similarly to the case with $\ell^{\infty}$ one sees that we can conclude
that every metric functional of this space is not only nonconstant
but in fact unbounded.
\end{proof}

\section{Application to bounded linear operators}

Recall first that we already noticed one application to norm-preserving
linear operators, showing a new result that defies the standard counterexamples,
and strong enough to deduce the most classical version of the mean
ergodic theorem.

The Invariant Subspace Problem (ISP) asks: does every bounded linear
operator $T$ of a separable complex Hilbert space of dimension at
least two, admit a nontrivial invariant closed subspace? This is a
famous and longstanding open problem, see for example the book \cite{CP11}
devoted to it. For the purpose of investigating invariant linear subspaces,
we may add a multiple of the identity map or multiply the operator
by a scalar, since these operation do not change any invariant subspace.
Therefore we may assume that the operator $g=T$ is invertible.

Note that by applying the operator $g$ iteratively on the basepoint
$x_{0}=Id$ we get 
\[
\tau(g)=\lim_{n\rightarrow\infty}\left|\log\left\Vert g^{n}g^{*n}\right\Vert ^{1/n}\right|
\]
this is semipositive and finite (since $g$ is an invertible bounded
operator). It is $0$ if and only if the spectrum of $g$ is contained
in the unit circle (for example when the operator is unitary). As
remarked in section 2 above $d(g)=\tau(g)$ (note that $d$ makes
sense in operator theory but is an invariant perhaps not much considered).
In view of Proposition \ref{prop:Posfunctional} we thus get from
Theorem \ref{thm:main} a nontrivial metric functional $h$ of $\mathrm{Pos}$
which is fixed by $g$
\[
g(h)=h.
\]

This proves Theorem \ref{thm:ISP}. It is clearly a statement in the
direction of the ISP, but how does it actually relate to closed invariant
subspaces? As already remarked when $g$ fixes a point inside $\mathrm{Pos}$
it is unitary and thus the spectral theorem applies and settles the
issue of invariant subspaces. If it fixes a point $s$ in the natural
boundary of $\mathrm{Pos}$ consisting of bounded semipositive operators,
then $g^{*}$ must preserve its kernel. Indeed, take $v\in\ker s$
then
\[
0=gs(v)=s(g^{*}v)
\]
showing that $g^{*}(\ker s)\subseteq\ker s.$ This is a closed subspace
and its closed orthogonal complement is invariant under $g$ and is
a nontrivial subspace. In general, one conceivable path forward is
to refine the theorem in this special case, and deduce that the metric
functional must be a Busemann point, and that as such it has an associated
closed linear subspace that must be invariant.

Another discussion of the invariant subspace problem from a metric
perspective without the space $\mathrm{Pos}$ can be found in \cite{GuK21}
which establishes some cases when there exist metric or linear functionals
$f$ of the Banach space such that $f(g^{n}(0))\leq0$ for all $n>0$
or $f(g(x))\leq f(x)$ for all vectors $x$.

\lyxaddress{Section de mathématiques, Université de Genève, Case postale 64,
1211 Genève, Switzerland; Mathematics department, Uppsala University,
Box 256, 751 05 Uppsala, Sweden.}

\begin{thebibliography}{BCFS22}
\bibitem[Al81]{Al81}Alspach, Dale E. A fixed point free nonexpansive
map. Proc. Amer. Math. Soc. 82 (1981), no. 3, 423\textendash 424.

\bibitem[AK22]{AK22}Avelin, B.; Karlsson, A. Deep limits and cut-off
phenomena for neural networks, to appear in J. Machine Learning Research,
2022

\bibitem[Ba18]{Ba18}Bacak, M. Old and new challenges in Hadamard
spaces, https://arxiv.org/pdf/1807.01355.pdf

\bibitem[BGM12]{BGM12}Bader, U.; Gelander, T.; Monod, N. A fixed
point theorem for L1 spaces. Invent. Math. 189 (2012), no. 1, 143\textendash 148.

\bibitem[BCFS22]{BCFS22}Bader, U.; Caprace, P.-E.; Furman, A.; Sisto,
A.Hyperbolic actions of higher-rank lattices come from rank-one factors,
https://arxiv.org/abs/2206.06431

\bibitem[Ba32]{Ba32}Banach, Stefan Théorie des opérations linéaires.
(French) Chelsea Publishing Co., New York, 1955. vii+254 pp.

\bibitem[BGS85]{BGS85}Ballmann, Werner; Gromov, Mikhael; Schroeder,
Viktor Manifolds of nonpositive curvature. Progress in Mathematics,
61. Birkhäuser Boston, Inc., Boston, MA, 1985. vi+263 pp.

\bibitem[B18]{B18}Basso, Giuliano Fixed point theorems for metric
spaces with a conical geodesic bicombing. Ergodic Theory Dynam. Systems
38 (2018), no. 5, 1642\textendash 1657.

\bibitem[BHL20]{BHL20}Bowen, Lewis; Hayes, Ben; Lin, Yuqing, A multiplicative
ergodic theorem for von Neumann algebra valued cocycles, Commun. Math.
Phys. 384: 2 (2021), 1291-1350.

\bibitem[BH99]{BH99}Bridson, Martin R.; Haefliger, André Metric spaces
of non-positive curvature. Grundlehren der mathematischen Wissenschaften,
319. Springer-Verlag, Berlin, 1999. xxii+643 pp. 

\bibitem[BM48]{BM48}Brodski\u{\i}, M. S.; Mil\textasciiacute man,
D. P. On the center of a convex set. (Russian) Doklady Akad. Nauk
SSSR (N.S.) 59, (1948). 837\textendash 840.

\bibitem[Br65a]{Br65a}Browder, Felix E. Fixed-point theorems for
noncompact mappings in Hilbert space. Proc. Nat. Acad. Sci. U.S.A.
53 (1965), 1272\textendash 1276.

\bibitem[Br65b]{Br65b}Browder, Felix E. Nonexpansive nonlinear operators
in a Banach space. Proc. Nat. Acad. Sci. U.S.A. 54 (1965), 1041\textendash 1044.

\bibitem[Ca11]{Ca11}Cantat, Serge. Sur les groupes de transformations
birationnelles des surfaces. Ann. of Math. 174(1) (2011), 299\textendash 340.

\bibitem[CL10]{CL10}Caprace, Pierre-Emmanuel; Lytchak, Alexander
At infinity of finite-dimensional CAT(0) spaces. Math. Ann. 346 (2010),
no. 1, 1\textendash 21.

\bibitem[CJ02]{CJ02}Carleson, L, Jones, P. Personal Reflections on
Analysis, EMS Newsletter, December 2002

\bibitem[CP11]{CP11}Chalendar, Isabelle; Partington, Jonathan R.
Modern approaches to the invariant-subspace problem. Cambridge Tracts
in Mathematics, 188. Cambridge University Press, Cambridge, 2011.
xii+285 pp.

\bibitem[Cl13]{Cl13}Clarke, Brian Geodesics, distance, and the CAT(0)
property for the manifold of Riemannian metrics. Math. Z. 273 (2013),
no. 1-2, 55\textendash 93. 

\bibitem[C20]{C20}Claassens, Floris The horofunction boundary of
the infinite dimensional hyperbolic space. Geom. Dedicata 207 (2020),
255\textendash 263.

\bibitem[CPR94]{CPR94}Corach, G.; Porta, H.; Recht, L. Convexity
of the geodesic distance on spaces of positive operators. Illinois
J. Math. 38 (1994), no. 1, 87\textendash 94.

\bibitem[DL15]{DL15}Descombes, Dominic; Lang, Urs Convex geodesic
bicombings and hyperbolicity. Geom. Dedicata 177 (2015), 367\textendash 384.

\bibitem[Ed64]{Ed64}Edelstein, Michael On nonexpansive mappings of
Banach spaces. Proc. Cambridge Philos. Soc. 60 (1964), 439\textendash 447.

\bibitem[Ei10]{Ei10}Eisner, Tanja Stability of operators and operator
semigroups. Operator Theory: Advances and Applications, 209. Birkhäuser
Verlag, Basel, 2010. viii+204 pp.

\bibitem[GV12]{GV12}Gaubert, Stéphane; Vigeral, Guillaume, A maximin
characterisation of the escape rate of nonexpansive mappings in metrically
convex spaces. Math. Proc. Cambridge Philos. Soc. 152 (2012), no.
2, 341\textendash 363.

\bibitem[GK20]{GK20}Gouëzel, S., Karlsson, A.: Subadditive and multiplicative
ergodic theorems. J. Eur. Math. Soc. (JEMS) 22(6), 1893\textendash 1915
(2020)

\bibitem[G065]{Go65}Göhde, D.: Zum Prinzip der kontraktiven Abbildung.
Math. Nachr. 30, 251\textendash 258 (1965).

\bibitem[Gr81]{Gr81}Gromov, M. Hyperbolic manifolds, groups and actions.
Riemann surfaces and related topics: Proceedings of the 1978 Stony
Brook Conference (State Univ. New York, Stony Brook, N.Y., 1978),
pp. 183-213, Ann. of Math. Stud., 97, Princeton Univ. Press, Princeton,
N.J., 1981.

\bibitem[Gu19]{Gu19}Gutiérrez, Armando W.; On the Metric Compactification
of Infinite-dimensional Spaces. Canad. Math. Bull. 62 (2019), no.
3, 491\textendash 507.

\bibitem[Gu20]{Gu20}Gutiérrez, Armando W. Characterizing the metric
compactification of Lp spaces by random measures. Ann. Funct. Anal.
11 (2020), no. 2, 227\textendash 243.

\bibitem[GuK21]{GuK21}Gutiérrez, Armando W.; Karlsson, Anders, Comments
on the cosmic convergence of nonexpansive maps. J. Fixed Point Theory
Appl. 23 (2021), no. 4, Paper No. 59, 10 pp.

\bibitem[HSW18]{HSW18}Haettel, Thomas; Schilling, Anne-Sofie;Walsh,
Cormac; Wienhard, Anna, Horofunction Compactifications of Symmetric
Spaces, https://arxiv.org/abs/1705.05026.

\bibitem[H07]{H07}Haglund, Frédéric, Isometries of CAT(0) cube complexes
are semi-simple, arXiv preprint arXiv:0705.3386 (2007).

\bibitem[Ha01]{Ha01}Handbook of metric fixed point theory. Edited
by William A. Kirk and Brailey Sims. Kluwer Academic Publishers, Dordrecht,
2001. xiv+703 pp.

\bibitem[HuO21]{HuO21}Huang, Jingyin; Osajda, Damian Helly meets
Garside and Artin. Invent. Math. 225 (2021), no. 2, 395\textendash 426.

\bibitem[KL18]{KL18}Kapovich, M. and Leeb, B., Finsler bordifications
of symmetric and certain locally symmetric spaces. Geometry and Topology,
22 (2018) 2533-2646.

\bibitem[K01]{K01}Karlsson, Anders, Non-expanding maps and Busemann
functions. Ergodic Theory Dynam. Systems 21 (2001), no. 5, 1447\textendash 1457.

\bibitem[K22]{K22}Karlsson, Anders; Elements of a metric spectral
theory. In: Dynamics, geometry, number theory\textemdash the impact
of Margulis on modern mathematics, 276\textendash 300, Univ. Chicago
Press, Chicago, IL, 2022.

\bibitem[K21a]{K21a}Karlsson, Anders, Hahn-Banach for metric functionals
and horofunctions. J. Funct. Anal. 281 (2021), no. 2, Paper No. 109030,
17 pp.

\bibitem[K21b]{K21b}Karlsson, Anders, From linear to metric functional
analysis. Proc. Natl. Acad. Sci. USA 118 (2021), no. 28, Paper No.
e2107069118, 5 pp.

\bibitem[KM99]{KM99}Karlsson, Anders; Margulis, Gregory A. A multiplicative
ergodic theorem and nonpositively curved spaces. Comm. Math. Phys.
208 (1999), no. 1, 107\textendash 123. 

\bibitem[KN04]{KN04}Karlsson, Anders; Noskov, Guennadi A. Some groups
having only elementary actions on metric spaces with hyperbolic boundaries.
Geom. Dedicata 104 (2004), 119\textendash 137

\bibitem[KL11]{KL11}Karlsson, Anders; Ledrappier, François Noncommutative
ergodic theorems. Geometry, rigidity, and group actions, 396\textendash 418,
Chicago Lectures in Math., Univ. Chicago Press, Chicago, IL, 2011. 

\bibitem[Ki65]{Ki65}Kirk, W. A. A fixed point theorem for mappings
which do not increase distances. Amer. Math. Monthly 72 (1965), 1004\textendash 1006. 

\bibitem[KoN81]{KoN81}Kohlberg, Elon; Neyman, Abraham Asymptotic
behavior of nonexpansive mappings in normed linear spaces. Israel
J. Math. 38 (1981), no. 4, 269\textendash 275.

\bibitem[La13]{La13}Lang, Urs Injective hulls of certain discrete
metric spaces and groups. J. Topol. Anal. 5 (2013), no. 3, 297\textendash 331.

\bibitem[L21]{L21}Lemmens, Bas, Horofunction compactifications of
symmetric cones under Finsler distances, arXiv:2111.12468.

\bibitem[LP21]{LP21}Lemmens, B. and Power, K. Horofunction compactifications
and duality, arXiv:2107.02710v3

\bibitem[LU21]{LU21}Lonjou, Anne; Urech, Christian, Actions of Cremona
groups on CAT(0) cube complexes, arXiv:2001.00783v2, 2021

\bibitem[MT18]{MT18}Maher, Joseph; Tiozzo, Giulio, Random walks on
weakly hyperbolic groups. J. Reine Angew. Math. 742 (2018), 187\textendash 239.

\bibitem[Ma81]{Ma81}Maurey, B. Points fixes des contractions sur
un convexe ferme de L1, Seminaire d'Analyse Fonctionelle, vol. 80-81,
Ecole Polytechnique Palaiseau, 1981.

\bibitem[N13]{N13}Navas, Andrés, An L1 ergodic theorem with values
in a non-positively curved space via a canonical barycenter map. Ergodic
Theory Dynam. Systems 33 (2013), no. 2, 609\textendash 623. 

\bibitem[Pe21]{Pe21}Pence, Zach, Metric Spectral Theory and the Invariant
Subspace Problem, Master thesis, Uppsala University, 2021

\bibitem[Ri02]{Ri02}Rieffel, M.A., Group $C^{*}$-algebras as compact
quantum metric spaces. Doc. Math. 7 (2002), 605\textendash 651

\bibitem[R82]{R82}Ruelle, David Characteristic exponents and invariant
manifolds in Hilbert space. Ann. of Math. (2) 115 (1982), no. 2, 243\textendash 290

\bibitem[Sch16]{Sch16}Schechtman, G., https://mathoverflow.net/questions/225597/generalizing-the-mazur-ulam-theorem-to-convex-sets-with-empty-interior-in-banach

\bibitem[W18]{W18}Walsh, Cormac Hilbert and Thompson geometries isometric
to infinite-dimensional Banach spaces. Ann. Inst. Fourier (Grenoble)
68 (2018), no. 5, 1831\textendash 1877.

\bibitem[YK41]{YK41}Yosida, Kôsaku; Kakutani, Shizuo Operator-theoretical
treatment of Markoff's process and mean ergodic theorem. Ann. of Math.
(2) 42 (1941), 188\textendash 228. 
\end{thebibliography}
\end{document}